\documentclass[12pt]{article}

\usepackage{amsmath,amsthm,amssymb}
\usepackage{lmodern}
\usepackage{hyperref}
\usepackage{microtype}
\usepackage{mathtools}
\usepackage[utf8]{inputenc}
\usepackage[T1]{fontenc}
\usepackage{booktabs}

\newtheorem{theorem}{Theorem}[section]
\newtheorem{fact}[theorem]{Fact}
\newtheorem{lemma}[theorem]{Lemma}

\theoremstyle{definition}

\newtheorem{remark}[theorem]{Remark}

\newcommand{\Met}{\operatorname{Met}}
\newcommand{\Fsigma}{F_{\!\sigma}}
\newcommand{\Gdelta}{G_{\!\delta}}
\newcommand{\Cont}{\mathfrak{c}}

\title{On Ishiki's Conjecture:\\
       $\Met(D)$ Is Not Completely Metrizable for
       $\lvert D\rvert=\aleph_1$}
\author{Tomoki Uda\thanks{Nanzan University.
  Email: \texttt{uda0@nanzan-u.ac.jp}.
  ORCID: \href{https://orcid.org/0000-0002-6073-1572}{0000-0002-6073-1572}.}}
\date{}

\begin{document}
\maketitle

\begin{abstract}
  For a discrete topological space $D$, let $\Met(D)$ denote the set of
  metrics on $D$ that are compatible with the discrete topology, equipped
  with the topology induced by the supremum distance. Ishiki's
  Conjecture~5.1 asserts that $\Met(D)$ is not completely metrizable when
  $\lvert D\rvert = \aleph_1$. We prove this in ZFC by constructing a
  set $A \subseteq [0,1]$ of cardinality $\aleph_1$ that is not
  $\Fsigma$ and embedding its complement as a closed subspace of $\Met(D)$.
  The proof has also been formalised in Lean~4.
\end{abstract}

% ------------------------------------------------------------------
\section{Introduction}
\label{sec:intro}

Let $D$ be an infinite discrete space, and let $\Met(D)$ denote the set
of metrics on $D$ compatible with its topology. We equip $\Met(D)$ with
the topology induced by the supremum distance
\[
  \rho(d, d') = \sup_{x,y\in D} \lvert d(x,y) - d'(x,y)\rvert ,
\]
which may take the value $+\infty$ for unbounded $d,d'$; this does not
affect complete metrizability, which is a property of a topology and not
of any particular compatible metric witnessing it (see
Remark~\ref{rem:extended-rho}).

Koshino~\cite{Koshino2024} proved that a separable metrizable space $X$
is $\sigma$-compact if and only if the space of bounded admissible
metrics on $X$, with the supremum metric, is completely metrizable (the
corollary in \cite{Koshino2024}). Ishiki~\cite{Ishiki2024} states the
corresponding characterisation in terms of $\Met(X)$: for separable $X$,
$\Met(X)$ is completely metrizable if and only if $X$ is $\sigma$-compact.
Conjecture~5.1 of \cite{Ishiki2024} is the
resulting question for the least uncountable case:

\begin{quote}
  \emph{If $\lvert D\rvert = \aleph_1$, then $\Met(D)$ is not completely
  metrizable.}
\end{quote}

We prove this conjecture. The source paper also
raises several further questions about $\Met(D)$ for other cardinalities
and other topological properties; those are outside the scope of this
paper, which addresses only Conjecture~5.1 as stated above.

The proof strategy is the following. Working in ZFC, we first construct
a set $A \subseteq [0,1]$ with
$\lvert A\rvert = \aleph_1$ that is not an $\Fsigma$ subset of $[0,1]$
(Lemma~\ref{lem:A-exists}). Its complement $A^\complement =
[0,1]\setminus A$ is then not a $\Gdelta$ subset of $[0,1]$, hence not
completely metrizable as a subspace of $[0,1]$. We then build an explicit
family $(d_t)_{t\in[0,1]}$ of metrics on a discrete space $D$ of
cardinality $\aleph_1$ that lie in $\Met(D)$ exactly for $t \in
A^\complement$ (Lemma~\ref{lem:star-metric}), and show that $t \mapsto
d_t$ restricts to a bi-Lipschitz embedding of $A^\complement$ onto a
\emph{closed} subset of $\Met(D)$ (Lemma~\ref{lem:closed-image}).
If $\Met(D)$ were completely metrizable, then the image
$\varphi[A^\complement]$, being closed in $\Met(D)$, would also be
completely metrizable, and hence so would $A^\complement$.
This contradicts the fact that $A^\complement$ is not completely
metrizable and proves Theorem~\ref{thm:conj51}.

\section{Preliminaries}
\label{sec:prelim}

\subsection{Notation}

Throughout, $\omega_1$ denotes the first uncountable ordinal, and
$\aleph_1 = \lvert\omega_1\rvert$ the first uncountable cardinal.
$\Cont \coloneqq 2^{\aleph_0}$ denotes the cardinality of the continuum; by
Cantor's theorem $\aleph_0 < \Cont$, and since $\aleph_1$ is the least
uncountable cardinal, $\aleph_1 \le \Cont$ in ZFC.
We write $\Fsigma$ (resp.\ $\Gdelta$) for the usual classes of
countable unions of closed sets (resp.\ countable intersections of open
sets), relative to whichever ambient space is under discussion. For $S
\subseteq [0,1]$ we write $S^\complement \coloneqq [0,1] \setminus S$;
complements are always taken relative to $[0,1]$. A topological space is
\emph{completely metrizable} if it is homeomorphic to a complete metric
space; this is weaker than \emph{Polish}, which additionally requires
separability.

\subsection{Background results}
\label{subsec:background}

The proof uses two classical facts about complete metrizability:
Fact~\ref{thm:background-gdelta} is applied to the ambient space
$[0,1]$, and Fact~\ref{thm:background-closed} to $\Met(D)$.

\begin{fact}[Alexandrov--Lavrentiev]
  \label{thm:background-gdelta}
  Let $X$ be a metrizable space and $S \subseteq X$. If $S$ is
  completely metrizable in the subspace topology, then $S$ is a $\Gdelta$
  subset of $X$.
\end{fact}
This is the forward direction of the classical $\Gdelta$-characterisation
of complete metrizability, stated in this generality in
\cite[Theorem~3.11]{Kechris1995}; see also
\cite[Theorems~4.3.23--4.3.24]{Engelking1989} and
\cite[Theorem~24.12]{Willard2004}.

\begin{fact}
  \label{thm:background-closed}
  A closed subspace of a completely metrizable space is completely metrizable.
\end{fact}
% \begin{fact}
%   Let $X$ be a completely metrizable space and $Y \subseteq X$ a closed
%   subset. Then $Y$, in the subspace topology, is completely metrizable.
% \end{fact}
% \begin{proof}
%   Fix a complete metric $d$ on $X$ inducing its topology. The restriction
%   $d|_{Y\times Y}$ induces the subspace topology on $Y$, so it suffices to
%   show that $d|_{Y\times Y}$ is complete. Let $(x_n)_n$ be a Cauchy
%   sequence in $(Y, d|_Y)$. It is also a Cauchy sequence in $(X,d)$, so by
%   completeness of $d$ it converges to some $x \in X$. In any metric space,
%   a closed set is closed under limits of sequences of its own points, so
%   $x \in Y$. Hence $(x_n)_n$ converges to $x$ within $(Y,d|_Y)$, which is
%   therefore complete, and $Y$ is completely metrizable.
% \end{proof}
Fact~\ref{thm:background-closed} is the only fact about complete
metrizability that we apply to $\Met(D)$.

\begin{remark}[The supremum distance is an extended metric]
  \label{rem:extended-rho}
  For unbounded metrics, $\rho(d,d')$ may be infinite. This causes no
  difficulty: $\min(\rho,1)$ induces the same topology as $\rho$ and is an
  ordinary metric. Thus $\Met(D)$ is metrizable, and complete metrizability
  depends only on its topology. In any case, every metric constructed in
  Section~\ref{sec:main} takes values in $[0,2]$, so all instances of
  $\rho$ appearing there are finite.
\end{remark}

\section{Main Construction}
\label{sec:main}

\subsection{A non-$\Fsigma$ set of cardinality $\aleph_1$}
\label{subsec:A}

\begin{lemma}
  \label{lem:closed-card}
  If $F \subseteq [0,1]$ is closed and $\lvert F\rvert < \Cont$, then $F$
  is countable.
\end{lemma}
\begin{proof}
  We prove the contrapositive. Let $F$ be closed and uncountable. Being
  closed in $[0,1]$, $F$ is itself Polish, and by the Cantor--Bendixson
  theorem \cite[Theorem~6.4]{Kechris1995} it splits as a disjoint union
  $F = P \sqcup C$ of a perfect set $P$ and a countable set $C$. Since $F$
  is uncountable and $C$ is countable, $P \ne \emptyset$, and a nonempty
  perfect Polish space has cardinality $\Cont$
  \cite[Theorem~6.2, Corollary~6.3]{Kechris1995}; hence
  $\lvert F\rvert \ge \lvert P\rvert = \Cont$.
\end{proof}

\begin{lemma}
  \label{lem:A-exists}
  There exists $A \subseteq [0,1]$ with $\lvert A\rvert = \aleph_1$
  such that $A$ is not an $\Fsigma$ set.
\end{lemma}

\begin{proof}
  We split on whether the continuum hypothesis $\mathsf{CH}$
  ($\aleph_1 = \Cont$) holds.

  \medskip
  \noindent\textbf{Case 1: $\aleph_1=\Cont$.}
  Let $A=[0,1]\setminus\mathbf{Q}$. Since $[0,1]$ has cardinality $\Cont$
  and $\mathbf{Q}\cap[0,1]$ is countable, $|A|=\Cont=\aleph_1$.

  Suppose for contradiction that $A=\bigcup_n F_n$, where each $F_n$ is
  closed in $[0,1]$. Since $\mathbf{Q}\cap[0,1]$ is dense and $F_n\subseteq A$,
  each $F_n$ has empty interior and is therefore nowhere dense. The set
  $\mathbf{Q}\cap[0,1]$ is also meagre. Hence $[0,1]$ would be meagre in
  itself, contradicting the Baire category theorem.

  \medskip
  \noindent\textbf{Case 2: $\aleph_1<\Cont$.}
  Choose $A\subseteq[0,1]$ with $|A|=\aleph_1$. Suppose for contradiction
  that $A=\bigcup_n F_n$, where each $F_n$ is closed in $[0,1]$. Since
  $|F_n|\le|A|=\aleph_1<\Cont$, Lemma~\ref{lem:closed-card} implies that
  every $F_n$ is countable. Hence $A$ is countable, contradicting $|A|=\aleph_1$.

  In either case, $A$ is not $\Fsigma$.
\end{proof}

\begin{remark}
  To avoid the case split on $\mathsf{CH}$, one may instead use a
  Bernstein set: a set $B \subseteq [0,1]$ such that $B$ and
  $B^\complement$ both meet every uncountable closed subset of $[0,1]$
  exists by transfinite recursion and has cardinality $\Cont$
  \cite[p.~24]{Oxtoby1980}. Take $A \subseteq B$ with $\lvert A\rvert =
  \aleph_1$. If $A = \bigcup_n F_n$ with each $F_n$ closed, then some
  $F_n$ is an uncountable closed set contained in $B$, contradicting
  that $B^\complement$ meets every uncountable closed set.
\end{remark}

\subsection{A parametrised family of metrics}
\label{subsec:star}

Fix, for the remainder of the paper, the set $A$ from
Lemma~\ref{lem:A-exists} together with an injective enumeration $A =
\{a_\alpha : \alpha < \omega_1\}$ (possible since $\lvert A\rvert =
\aleph_1$).

\begin{lemma}
  \label{lem:star-metric}
  There is a discrete space $D$ with $\lvert D\rvert = \aleph_1$ and a
  family $(d_t)_{t \in [0,1]}$ of metrics on $D$ such that $d_t \in
  \Met(D)$ holds if and only if $t \in A^\complement$.
\end{lemma}

A bijection $f \colon D\to D'$ induces an isometric isomorphism between
$\Met(D)$ and $\Met(D')$ by relabelling the two arguments of each metric.
Since the isometry type of $\Met(D)$ thus depends only on $|D|$, it
suffices to consider the space $D$ constructed below, here and in the
proof of Theorem~\ref{thm:conj51}.

\begin{proof}
  \noindent\textbf{Construction.} Let
  \[
    D \coloneqq \{p,q\} \sqcup \{y_\alpha : \alpha<\omega_1\} \sqcup
    \{x_{\alpha,n} : \alpha < \omega_1,\ n \in \mathbf{N}\},
  \]
  a disjoint union with $\lvert D\rvert = 2 + \aleph_1 + \aleph_1\cdot
  \aleph_0 = \aleph_1$, carrying the discrete topology. For $\alpha <
  \omega_1$ let $C_\alpha \coloneqq \{y_\alpha\} \cup \{x_{\alpha,n} :
  n\in\mathbf{N}\}$. Then
  \[
    D = \{p,q\} \sqcup \bigsqcup_{\alpha<\omega_1} C_\alpha
  \]
  is a partition of $D$. For $t \in [0,1]$, $\alpha < \omega_1$,
  $n \in \mathbf{N}$, set
  \[
    s_n(\alpha,t) \coloneqq \min\left(1,\ \lvert t-a_\alpha\rvert +
    \dfrac{1}{n+1}\right)
    \in (0,1].
  \]
  Define $d_t \colon D\times D\to[0,2]$ by
  \[
    d_t(u,v)=
    \begin{cases}
      0, & u=v,\\
      1+t, & \{u,v\}=\{p,q\},\\
      s_n(\alpha,t), & \{u,v\}=\{y_\alpha,x_{\alpha,n}\},\\
      s_n(\alpha,t)+s_m(\alpha,t),
      & \{u,v\}=\{x_{\alpha,n},x_{\alpha,m}\},\ n\ne m,\\
      1, & \text{otherwise}.
    \end{cases}
  \]
  The range of $d_t$ lies in $[0,2]$ since $0 < s_n \le 1$ and $0\le t\le 1$.

  \medskip\noindent\textbf{Claim (C1).} $d_t$ is a metric, for every $t \in [0,1]$.
  Definiteness and symmetry are immediate. For the triangle inequality
  $d_t(u,w)\le d_t(u,v)+d_t(v,w)$, the case where two of $u,v,w$ coincide is
  trivial, so assume $u,v,w$ are pairwise distinct, and abbreviate $s_n \coloneqq
  s_n(\alpha,t)$ once $\alpha$ is fixed. We use two facts: every distance is
  at most $2$, and two points in distinct parts of the partition are at distance
  exactly $1$. Split on how $\{u,v,w\}$ meets the partition.

  \begin{itemize}
  \item \emph{All three in one part.} Since $\{p,q\}$ has only two
    elements, the part is some $C_\alpha$. If the three points are
    $y_\alpha$, $x_{\alpha,n}$, $x_{\alpha,m}$ ($n \ne m$), the required
    inequalities are $s_n+s_m \le s_n + s_m$ (through $y_\alpha$) and
    $s_n \le (s_n+s_m) + s_m$ (through $x_{\alpha,m}$). If the three points
    are $x_{\alpha,n},x_{\alpha,m},x_{\alpha,k}$ (distinct indices), then
    $s_n+s_m \le (s_n+s_k) + (s_k+s_m)$.
  \item \emph{Exactly two in one part $E$, one point $z$ outside $E$.} Then
    $d_t(u,z)=d_t(v,z)=1$ for the two points $u,v \in E$, so
    $d_t(u,v)\le 2 = d_t(u,z)+d_t(z,v)$ and $d_t(u,z)=1\le d_t(u,v)+1 =
    d_t(u,v)+d_t(v,z)$.
  \item \emph{Three distinct parts.} All three pairwise distances equal
    $1$, and $1\le 1+1$.
  \end{itemize}
  These cases exhaust all possible configurations of three points with respect
  to the partition. Hence $d_t$ satisfies the triangle inequality and is a metric.

  \medskip\noindent\textbf{Claim (C2).} $t \in A^\complement$ implies $d_t \in \Met(D)$.
  A metric induces the discrete topology on $D$ exactly when every point is
  isolated, i.e. $\delta(u) \coloneqq \inf_{v\ne u} d_t(u,v) > 0$ for every $u \in
  D$. Fix $t \in A^\complement$.
  \begin{itemize}
  \item $\delta(p)=\delta(q)=1$: the distance from $p$ (resp.\ $q$) to any
    point other than $q$ (resp.\ $p$) is $1$, and $d_t(p,q)=1+t\ge 1$.
  \item For $x_{\alpha,n}$: its distances are $s_n(\alpha,t)$
    (to $y_\alpha$), $s_n(\alpha,t)+s_m(\alpha,t) \ge s_n(\alpha,t)$ (to
    $x_{\alpha,m}$, $m \ne n$), and $1$ (outside $C_\alpha$); since $s_n
    \le 1$, $\delta(x_{\alpha,n}) = s_n(\alpha,t) \ge 1/(n+1)>0$.
  \item For $y_\alpha$: its distances are $s_n(\alpha,t)$
    ($n\in\mathbf{N}$), and $1$ outside $C_\alpha$.
    Since $r\mapsto\min(1,r)$ is continuous and nondecreasing and
    $\lvert t-a_\alpha\rvert+1/(n+1) \downarrow \lvert t-a_\alpha\rvert$ as
    $n\to\infty$, $\inf_n s_n(\alpha,t) = \min(1,\lvert t-a_\alpha\rvert)$,
    so $\delta(y_\alpha)=\min(1,\lvert t-a_\alpha\rvert)$. As $t \in
    A^\complement$ and $a_\alpha \in A$, $t \ne a_\alpha$, so
    $\delta(y_\alpha)>0$.
  \end{itemize}
  Every point is isolated, so $d_t$ induces the discrete topology; with
  (C1), $d_t \in \Met(D)$.

  \medskip\noindent\textbf{Claim (C3).} $t \in A$ implies $d_t \notin \Met(D)$.
  Let $t=a_\alpha\in A$. By (C1), $d_t$ is a metric on $D$. Since
  \[
    d_t(y_\alpha,x_{\alpha,n})=\frac{1}{n+1}\to 0
    \qquad \text{(as $n\to\infty$)},
  \]
  the point $y_\alpha$ is not isolated. Hence $d_t\notin\Met(D)$.

  Claims (C2) and (C3) prove Lemma~\ref{lem:star-metric}.
\end{proof}

% \begin{remark}
%   \label{rem:star-name}
%   By the equality case in (C1), the distance between two points
%   $x_{\alpha,n}$ and $x_{\alpha,m}$ of a common $C_\alpha$ is realised by
%   the path through $y_\alpha$, so each $(C_\alpha, d_t)$ is isometric to
%   a metric star. This geometric feature is not used in the sequel.
% \end{remark}

\section{Proof of the Main Theorem}
\label{sec:closed-image}

We continue to use the $A$, $D$, and $(d_t)_{t\in[0,1]}$ fixed in
Section~\ref{subsec:star}. Give $A^\complement$ the subspace topology
from $[0,1]$.

\begin{lemma}
  \label{lem:closed-image}
  The map $\varphi \colon A^\complement \to \Met(D)$, $\varphi(t) \coloneqq d_t$, is a
  bi-Lipschitz embedding, and its image $\varphi[A^\complement]$ is
  closed in $\Met(D)$.
\end{lemma}

\begin{proof}
  \noindent\textbf{Claim (C4) (Bi-Lipschitz estimate).} For all $s,t\in[0,1]$,
  \[
    \lvert t-s\rvert \le \rho(d_t,d_s) \le  2\lvert t-s\rvert .
  \]
  \emph{Lower bound.} Comparing values on the pair $\{p,q\}$,
  $\lvert d_t(p,q)-d_s(p,q)\rvert = \lvert(1+t)-(1+s)\rvert = \lvert
  t-s\rvert$, and $\rho(d_t,d_s)$ is at least this one value.

  \emph{Upper bound.} It suffices to bound $\lvert d_t(u,v)-d_s(u,v)\rvert
  \le 2\lvert t-s\rvert$ for every pair $\{u,v\}$ and take the supremum. On
  $\{p,q\}$ the difference is $\lvert t-s\rvert$. The map $r
  \mapsto \min(1,r)$ is $1$-Lipschitz, and $r \mapsto \lvert r-a_\alpha\rvert
  + 1/(n+1)$ is $1$-Lipschitz by the reverse triangle inequality, so their
  composite $r \mapsto s_n(\alpha,r)$ is $1$-Lipschitz: $\lvert
  s_n(\alpha,t)-s_n(\alpha,s)\rvert \le \lvert t-s\rvert$. Hence for a
  pair $\{y_\alpha, x_{\alpha,n}\}$ the difference is at most $\lvert
  t-s\rvert \le 2\lvert t-s\rvert$, and for a pair
  $\{x_{\alpha,n},x_{\alpha,m}\}$ it is at most $2\lvert t-s\rvert$ by the
  triangle inequality on two such terms. For pairs meeting two distinct
  parts of the partition the value is constantly $1$, so the difference is
  $0$. This proves the upper bound, hence (C4).

  Since $\varphi$ is injective by the lower bound and both Lipschitz
  estimates hold, $\varphi$ is a bi-Lipschitz bijection from
  $A^\complement$ onto $\varphi[A^\complement]$ (with the subspace topology
  from $\Met(D)$), in particular a homeomorphism onto its image.

  \medskip\noindent\textbf{Claim (C5) (The image is closed).} The topology of $\Met(D)$ is given
  by the finite metric $\min(\rho,1)$, so closure points are sequential
  limits. Let $d' \in \Met(D)$ be a limit point of
  $\varphi[A^\complement]$, and choose $t_k \in A^\complement$ with
  $\rho(d_{t_k},d')\to 0$. We show $d' \in \varphi[A^\complement]$.

  \emph{Recovering the parameter.} $\lvert d_{t_k}(p,q)-d'(p,q)\rvert \le
  \rho(d_{t_k},d')\to 0$ gives $1+t_k \to d'(p,q)$, i.e.\ $t_k \to t$ for $t \coloneqq
  d'(p,q)-1$. Each $t_k \in [0,1]$ and $[0,1]$ is closed in $\mathbf{R}$, so
  $t \in [0,1]$, and by (C1) $d_t$ is a well-defined metric on $D$.

  \emph{Identifying the limit.} Fix $u,v \in D$. By (C4),
  $\rho(d_{t_k},d_t) \le 2\lvert t_k-t\rvert \to 0$, so
  \[
    \lvert d'(u,v)-d_t(u,v)\rvert \le \rho(d',d_{t_k}) +
    \rho(d_{t_k},d_t) \longrightarrow 0 \quad (k \to \infty),
  \]
  while the left side does not depend on $k$; hence $d'(u,v)=d_t(u,v)$ for
  all $u,v$, i.e.\ $d' = d_t$.

  \emph{Excluding the boundary case.}
  Since $d'=d_t$ and $d'\in\Met(D)$, we have $d_t\in\Met(D)$.
  If $t\in A$, however, (C3) gives $d_t\notin\Met(D)$, a contradiction.
  Hence $t \notin A$, i.e.\ $t \in A^\complement$, and $d' = d_t
  = \varphi(t) \in \varphi[A^\complement]$.

  Thus $\varphi[A^\complement]$ is closed in $\Met(D)$, proving
  Lemma~\ref{lem:closed-image}. The essential point is to exclude
  the case $t\in A$, where $d_t$ is a metric but does not induce
  the discrete topology. This is precisely what (C3) rules out.
\end{proof}

\begin{theorem}[Ishiki's Conjecture~5.1]
  \label{thm:conj51}
  If $\lvert D\rvert = \aleph_1$, then $\Met(D)$, with the supremum-metric
  topology, is not completely metrizable.
\end{theorem}

\begin{proof}
  By Lemma~\ref{lem:A-exists}, choose $A\subseteq[0,1]$ with
  $|A|=\aleph_1$ that is not $\Fsigma$. Then $A^\complement$ is not
  $\Gdelta$ in $[0,1]$, and hence is not completely metrizable by
  Fact~\ref{thm:background-gdelta}.
  By Lemmas~\ref{lem:star-metric} and \ref{lem:closed-image},
  $A^\complement$ is homeomorphic to the closed subspace
  $\varphi[A^\complement]$ of $\Met(D)$. If $\Met(D)$ were completely
  metrizable, then so would be $\varphi[A^\complement]$, and consequently
  $A^\complement$, a contradiction.
\end{proof}

\section{Concluding Remarks}
\label{sec:concluding}

Koshino \cite{Koshino2024} proved that, for a separable metrizable space
$X$, the space of bounded admissible metrics on $X$, equipped with the
supremum metric, is completely metrizable if and only if $X$ is
$\sigma$-compact. Ishiki \cite{Ishiki2024} stated the corresponding
characterisation for $\Met(X)$ and formulated Conjecture~5.1 for a
discrete space of cardinality $\aleph_1$. Koshino's proof also uses a
one-parameter family of metrics to embed a rational subspace
$Q'\subseteq[0,1]$ as a closed subspace of the space of bounded
admissible metrics. Our construction has a similar form, with
$A^\complement$ in place of $Q'$.

No additional set-theoretic assumption, such as the continuum
hypothesis, is required in the proof of Theorem~\ref{thm:conj51}. The
proof has also been formalised in Lean~4 using Mathlib.

% ------------------------------------------------------------------
\section*{Acknowledgements}

The author is grateful to Yoshito Ishiki for helpful comments on an
earlier version of this manuscript.

% ------------------------------------------------------------------
\bibliographystyle{plain}
\bibliography{references}

\end{document}